\documentclass [11pt,] {article}
\usepackage{amsmath,amssymb}
\usepackage{amsthm,amsfonts,amscd,epsfig,lineno}
\usepackage{xcolor}
\usepackage{multirow}
\usepackage{cite,epic}
\usepackage{adjustbox}
\usepackage{graphicx}
\usepackage[left=2cm]{geometry}
\usepackage{lscape}

\usepackage{amsthm}

\usepackage{makecell}
\usepackage{float}
\usepackage{booktabs}
\usepackage{amsmath, amssymb, amsthm}

\newtheorem{theorem}{Theorem}[section]
\newtheorem{lemma}[theorem]{Lemma}

\newtheorem{definition}[theorem]{Definition}

\newtheorem{construction}[theorem]{Construction}

\usepackage{booktabs}
\usepackage{multirow}
\usepackage{makecell}
\usepackage{array}
\usepackage{float}

\usepackage{caption}
\newcount\refno
\usepackage{cite}
\numberwithin{equation}{section}

\usepackage[unicode,hidelinks]{hyperref}

\title{Orthogonal Quantum Latin Squares with Maximal Cardinality
from Full-Rank Factorizations}
\author {   \footnotesize     Mingzhen Lv,   Yuyuan Zhang,   Haitao Cao\footnote{Corresponding author. \; E-mail address: caohaitao@njnu.edu.cn.  }\\
	\scriptsize    School of Mathematical Sciences, Ministry of Education Key Laboratory for NSLSCS,\\ \scriptsize Nanjing Normal University, Nanjing 210023, China. }
\date{}
\begin{document}
\parindent=0.5cm
\baselineskip=0.6cm
\maketitle
\begin{abstract}
	In this paper, we investigate the existence of pairs of orthogonal
	quantum Latin squares, each with maximal cardinality.
	Using characters and full-rank factorizations of finite abelian groups,
	together with a product construction, we establish two infinite
	families of orders for which such pairs exist.
\end{abstract}


\section{Introduction}
A \emph{quantum Latin square} of order $n$, denoted by QLS$(n)$,
is an $n\times n$ array of unit column vectors in $\mathbb C^n$
such that each row and each column forms an orthonormal basis of
$\mathbb C^n$.
Musto and Vicary~\cite{MV16} introduced quantum Latin squares as a
quantum-theoretic generalization of classical Latin squares.
They are connected with unitary error bases, mutually unbiased bases,
and $k$-uniform states~\cite{MV16,M17,GRMZ}.
For further research on quantum Latin squares,
see~\cite{AR,CLV,DNV,KSS,RV,NP}.

The \emph{cardinality} $c$ of a QLS$(n)$ is the number of distinct
quantum states among its entries, counted up to a global phase~\cite{PWRBZ}.
Thus $|\phi\rangle$ and $e^{i\theta}|\phi\rangle$, with $\theta\in\mathbb R$,
represent the same state for this count, and $n\leq c\leq n^2$.
A QLS$(n)$ has \emph{maximal cardinality} if $c=n^2$.
A classical Latin square of order $n$ is an $n\times n$ array on $n$
symbols in which each symbol occurs once in every row and column.
Replacing its symbols $0,\ldots,n-1$ with the computational basis
vectors $|0\rangle,\ldots,|n-1\rangle$ gives a QLS$(n)$ of cardinality
$n$. A QLS$(n)$ is called \emph{classical} if its cardinality is $n$;
otherwise, it is called \emph{non-classical}.
Zang et al.~\cite{ZZTS} proved that a QLS$(n)$ with maximal cardinality
exists for every $n\geq4$.
The possible cardinalities of quantum Latin squares have also been
studied in~\cite{ZWJ,ZJ,ZLC}.

\begin{definition}[{\cite{GRMZ,MV19}}]
Let $P=(|p_{ij}\rangle)$ and $Q=(|q_{ij}\rangle)$ be two QLS$(n)$s.
They are \emph{orthogonal} if
$
\{|p_{ij}\rangle\otimes|q_{ij}\rangle:0\leq i,j\leq n-1\}
$
forms an orthonormal basis of $\mathbb C^n\otimes\mathbb C^n$.
\end{definition}

A set of $r$ mutually orthogonal QLS$(n)$s is denoted by
$r$-$\mathrm{MOQLS}(n)$.
Two classical Latin squares are orthogonal if and only if their
corresponding quantum Latin squares are orthogonal.
A set of QLS$(n)$s is called \emph{non-classical} if at least one
of them is non-classical.

Musto and Vicary~\cite{MV19} constructed a non-classical orthogonal
pair of order $9$ and proved that any $r$-$\mathrm{MOQLS}(n)$ satisfies
$r \le n-1$.
For order $6$, Ball and Simoens~\cite{BS6} proved that no orthogonal
pair exists under the definition above. The solution of Rather
et al.~\cite{RBBRLZ} concerns the broader setting allowing entangled
entries. For further existence results and bounds on non-classical mutually
orthogonal $\mathrm{QLS}$s, see~\cite{ZFT,HZHT,HL,BSL}.

The results cited above on non-classical orthogonal families concern
sets with at least one non-classical member.
We focus on orthogonal pairs in which both
members have maximal cardinality.
In this paper, an $\mathrm{MCOQLS}(n)$ denotes a pair of orthogonal
QLS$(n)$s, each with maximal cardinality.
Our main result is the following.
\begin{theorem}\label{thm:main}
There exists an $\mathrm{MCOQLS}(n)$ in each of the following cases:
\begin{enumerate}
\item[(i)] $n=q_1q_2k$, where each of $q_1,q_2$ is either an odd
prime power at least $9$ or $16$,
and $k$ is a positive integer other than $2$ and $6$;
\item[(ii)] $n=64m$, where $m$ is any positive integer.
\end{enumerate}
\end{theorem}

Section~\ref{sec:prelim} recalls the preliminary results and presents
a product construction.
Section~\ref{sec:character} presents constructions from full-rank
factorizations.
Section~\ref{sec:existence} establishes the required existence results
and uses the product construction to prove Theorem~\ref{thm:main}.
Section~\ref{sec:conclusion} concludes the paper and presents two open questions
on the existence of MCOQLSs.

\section{Preliminaries}\label{sec:prelim}
For a positive integer $n$, let $[n]=\{0,\ldots,n-1\}$.
For a prime power $q$, let $\mathbb F_q$ denote the field with $q$
elements.
For a group $G$, let $\mathrm{id}_G$ denote the identity map on $G$.

We recall characters, normalized full-rank factorizations of
finite abelian groups and orthogonal orthomorphisms. Then we present a
product construction for $\mathrm{MCOQLS}$s.

\begin{definition}
	Let $G$ be a finite abelian group, written additively.
	A character of $G$ is a map $\chi:G\to\mathbb C$ such that
	$|\chi(a)|=1$ and $\chi(a+b)=\chi(a)\chi(b)$ for all $a,b\in G$.
\end{definition}

The characters form a group $\widehat G$ under pointwise multiplication,
called the \emph{character group} of $G$.
Since $\widehat G$ is isomorphic to $G$, fix an isomorphism
$g\mapsto\chi_g$ from $G$ to $\widehat G$.
Then $\chi_{g+h}(a)=\chi_g(a)\chi_h(a)$,
$\chi_{-g}(a)=\overline{\chi_g(a)}$, and $\chi_0(a)=1$ for all
$g,h,a\in G$.
The character orthogonality relations give
$$
 \frac1{|G|}\sum_{a\in G}\overline{\chi_g(a)}\chi_h(a)
 =\delta_{g,h},\qquad g,h\in G.
$$
In particular, $\frac1{|G|}\sum_{a\in G}\chi_g(a)=\delta_{g,0}$.

\begin{definition}[\cite{Dinitz}]
Let $A,B$ be subsets of a finite abelian group $G$.
The pair $(A,B)$ is a factorization of $G$ if every element of $G$
has a unique representation $a+b$, where $a\in A$ and $b\in B$.
The factorization is normalized if $0\in A\cap B$, and it has
full rank if $\langle A\rangle=\langle B\rangle=G$, where
$\langle S\rangle$ denotes the subgroup generated by $S$.
\end{definition}
Every such factorization satisfies $|A||B|=|G|$.

\begin{lemma}\label{lem:automorphism-factorization}
Let $(A,B)$ be a normalized full-rank factorization of a finite abelian
group $G$, and let $\varphi\in\operatorname{Aut}(G)$.
Then $(\varphi(A),\varphi(B))$ is also a normalized full-rank
factorization of $G$.
\end{lemma}
\begin{proof}
For each $x\in G$, the element $\varphi^{-1}(x)$ has a unique
representation $a+b$ with $a\in A$ and $b\in B$.
Thus $x= \varphi(\varphi^{-1}(x)) =\varphi(a)+\varphi(b)$, where $\varphi(a)\in \varphi(A)$ and $\varphi(b)\in\varphi(B)$. Since $|\varphi(A)||\varphi(B)|=|A||B|=|G|$, the representation is unique.

Moreover,
$
\langle\varphi(A)\rangle=\varphi(\langle A\rangle)=G,
\langle\varphi(B)\rangle=\varphi(\langle B\rangle)=G,
$
and $0_G=\varphi(0_G)\in \varphi(A) \cap \varphi(B)$. Thus $(\varphi(A),\varphi(B))$ is also a normalized full-rank
factorization of $G$.
\end{proof}

\begin{definition}[{\cite{Evans}}]
An orthomorphism of a finite abelian group $G$
is a permutation $\theta:G\to G$
such that $\theta-\mathrm{id}_{G}$ is a permutation. It is \emph{normalized} if
$\theta(0)=0$. Two orthomorphisms $\theta,\eta$ are
\emph{orthogonal} if $\theta-\eta$ is a permutation.
\end{definition}

The following existence result will be used later.

\begin{lemma}[{\cite[Theorem~11]{FW}}]\label{lem:near-linear}
Let $q\geq9$ be an odd prime power and let $M$ be the subgroup of
nonzero squares in $\mathbb F_q$, or let $q=16$ and let $M$ be the
multiplicative subgroup of order $5$.
There exist distinct $a,b\in\mathbb F_q^\times$ and
$c\in\mathbb F_q\setminus\{0,1\}$ such that
$$
\theta(x)=
\begin{cases}
ax,&x\in M,\\
bx,&x\notin M,
\end{cases}
\qquad \eta(x)=cx
$$
are normalized orthogonal orthomorphisms of $(\mathbb F_q,+)$.
\end{lemma}

Before we present a product construction for obtaining such pairs of
larger orders, we recall two results needed for the construction.

\begin{lemma}[{\cite[Lemma 2.7]{ZLC}}]\label{lem:unitary-copies}
Let $P$ be a $\mathrm{QLS}(d)$, where $d\geq3$, and let $k$ be a positive integer.
There exist unitary matrices $U_i$, $i\in[k]$, such that the arrays
$U_iP$ have pairwise disjoint sets of quantum states, where $U_iP$ is obtained by applying $U_i$ to every entry of $P$.
\end{lemma}

\begin{lemma}[{\cite[Construction 3.9]{ZLC}}]\label{lem:product-cardinality}
Let $L=(L_{ij})$ be a classical Latin square of order $k$ with entries
in $[k]$. Let $P_i=(|p^i_{ab}\rangle)$, $i\in[k]$, be QLS$(d)$s whose
union contains $c$ distinct quantum states.
The array with entries $|L_{ij}\rangle\otimes|p^i_{ab}\rangle$, indexed
by $(i,a),(j,b)\in[k]\times[d]$, is a QLS$(kd)$ of cardinality $kc$.
\end{lemma}

Using the preceding two lemmas, we obtain the following product construction.

\begin{construction}\label{con:product}
Let $d\geq3$ and $k\geq1$ be integers.
Suppose the following conditions are satisfied:
\begin{enumerate}
\item[(i)] There exist two orthogonal classical Latin squares of order $k$;
\item[(ii)] There exists an $\mathrm{MCOQLS}(d)$.
\end{enumerate}
Then there exists an $\mathrm{MCOQLS}(kd)$.
\end{construction}
\begin{proof}
Let $P=(|p_{ab}\rangle)$ and $Q=(|q_{ab}\rangle)$ be two orthogonal
QLS$(d)$s, each with maximal cardinality.
Let $L=(L_{ij})$ and $M=(M_{ij})$ be two orthogonal classical Latin
squares of order $k$, with entries in $[k]$.
By Lemma~\ref{lem:unitary-copies}, choose unitary matrices $U_i$, $i\in[k]$, such that the arrays $U_iP$
have pairwise disjoint sets of quantum states. Similarly, choose
$V_i$, $i\in[k]$, such that the arrays $V_iQ$ have the same property.
Define arrays $\widetilde P$ and $\widetilde Q$ by
$$
\begin{aligned}
 |\widetilde p_{(i,a),(j,b)}\rangle
 &=|L_{ij}\rangle\otimes U_i|p_{ab}\rangle,\quad
 |\widetilde q_{(i,a),(j,b)}\rangle
 &=|M_{ij}\rangle\otimes V_i|q_{ab}\rangle,
 \quad i,j\in[k], a,b\in[d].
\end{aligned}
$$
The arrays $U_iP$, $i\in[k]$, contain $kd^2$ distinct quantum
states in total, and the same holds for the arrays $V_iQ$.
By Lemma~\ref{lem:product-cardinality}, $\widetilde P$ and
$\widetilde Q$ are QLS$(kd)$s, each with maximal cardinality $(kd)^2$.

For two cells, orthogonality of $L,M$ and of $P,Q$ gives
$$
\begin{aligned}
 &\langle\widetilde p_{(i,a),(j,b)}|
          \widetilde p_{(i',a'),(j',b')}\rangle
  \langle\widetilde q_{(i,a),(j,b)}|
          \widetilde q_{(i',a'),(j',b')}\rangle\\
 &\quad=\delta_{L_{ij},L_{i'j'}}\delta_{M_{ij},M_{i'j'}}
  \langle p_{ab}|U_i^*U_{i'}|p_{a'b'}\rangle
  \langle q_{ab}|V_i^*V_{i'}|q_{a'b'}\rangle\\
 &\quad=\delta_{i,i'}\delta_{j,j'}
  \langle p_{ab}|p_{a'b'}\rangle\langle q_{ab}|q_{a'b'}\rangle\\
 &\quad=\delta_{i,i'}\delta_{j,j'}\delta_{a,a'}\delta_{b,b'}.
\end{aligned}
$$
Here $U_i^*$ and $V_i^*$ denote the conjugate transposes.
Hence $\widetilde P$ and $\widetilde Q$ are orthogonal.
\end{proof}

Bose, Shrikhande and Parker~\cite{BSP} established that two orthogonal
classical Latin squares of order $k$ exist for every $k\geq2$ other
than $2$ and $6$.
Thus Construction~\ref{con:product}, together with the trivial case
$k=1$, allows the order of such a pair to be multiplied by any positive
integer other than $2$ and $6$.

\section{Constructions from full-rank factorizations}\label{sec:character}
In this section, we first present a construction of MCOQLSs from a special class of normalized full-rank factorizations. We then give two methods for obtaining the required factorizations: one based on orthogonal orthomorphisms and the other on normalized full-rank factorizations equipped with suitable endomorphisms.

For a finite abelian group $G$, let $\pi_1,\pi_2:G\times G\to G$ be the coordinate projections defined by
$$
\pi_1(x,y)=x,\quad \pi_2(x,y)=y.
$$

\begin{construction}\label{con:projections}
Let $G$ be a finite abelian group of order $n$.
Suppose the following conditions are satisfied:
\begin{enumerate}
\item[(i)] There exists a normalized full-rank factorization $(A,B)$
of $G\times G$;
\item[(ii)] The restrictions of $\pi_1$ and $\pi_2$ to each of $A$ and
$B$ are bijections onto $G$.
\end{enumerate}
Then there exists an $\mathrm{MCOQLS}(n)$.
\end{construction}
\begin{proof}
Let $H=G\times G$. By (ii), we have $|A|=|B|=n$.
Index the computational bases of two copies of $\mathbb C^n$ by
$A$ and $B$, respectively. Define arrays $Q=(|q_{xy}\rangle)$ and
$R=(|r_{xy}\rangle)$, with $x,y\in G$, by
$$
 |q_{xy}\rangle
 =\frac1{\sqrt n}\sum_{a\in A}
   \chi_x(\pi_1a)\chi_y(\pi_2a)|a\rangle,\quad
 |r_{xy}\rangle
 =\frac1{\sqrt n}\sum_{b\in B}
   \chi_x(\pi_1b)\chi_y(\pi_2b)|b\rangle.
$$
Every entry of $Q$ is a unit vector, since
$$
 \langle q_{xy}|q_{xy}\rangle
 =\frac1n\sum_{a\in A}
   |\chi_x(\pi_1a)\chi_y(\pi_2a)|^2
 =\frac{|A|}{n}=1.
$$
The same holds for $R$.
In a fixed row or column of $Q$, the inner products of two entries are
$$
\begin{aligned}
 \langle q_{xy}|q_{xy'}\rangle
 &=\frac1n\sum_{a\in A}\chi_{y'-y}(\pi_2a)
 =\frac1n\sum_{g\in G}\chi_{y'-y}(g)=\delta_{y,y'},\\
 \langle q_{xy}|q_{x'y}\rangle
 &=\frac1n\sum_{a\in A}\chi_{x'-x}(\pi_1a)
 =\frac1n\sum_{g\in G}\chi_{x'-x}(g)=\delta_{x,x'}.
\end{aligned}
$$
Here we used the bijectivity of $\pi_1|_A$ and $\pi_2|_A$. The corresponding bijections
on $B$ give the same identities for $R$. Thus $Q$ and $R$ are QLS$(n)$s.

For two cells $(x,y)$ and $(x',y')$, define the character
$\phi:H\to\mathbb C$ by
$$
\phi(h)=\chi_{x'-x}(\pi_1h)\chi_{y'-y}(\pi_2h).
$$
The character $\phi$ is trivial if and only if $x=x'$ and $y=y'$. Indeed, if $\phi$ is trivial, then for every $g\in G$,
$
\phi(g,0)=\chi_{x'-x}(g)=1,
\phi(0,g)=\chi_{y'-y}(g)=1.
$
Hence $\chi_{x'-x}=\chi_{y'-y}=\chi_0$, and thus $x=x'$ and $y=y'$. The converse is immediate.

Suppose that $|q_{x'y'}\rangle=\lambda|q_{xy}\rangle$ for some
$|\lambda|=1$. Since $0_H\in A$, the coefficient of $|0_H\rangle$
is $1/\sqrt n$ in both vectors, so $\lambda=1$.
Comparing the remaining coefficients gives $\phi(a)=1$ for every
$a\in A$. Since $A$ generates $H$, the character $\phi$ is trivial
on $H$. Hence $x'=x$ and $y'=y$.
The same argument applies to $R$, since $0_H\in B$ and $B$ generates
$H$. Therefore, both arrays contain $n^2$ distinct quantum states.

The factorization property and character orthogonality give
$$
 \langle q_{xy}|q_{x'y'}\rangle
 \langle r_{xy}|r_{x'y'}\rangle
 =\frac1{n^2}\sum_{a\in A,\,b\in B}\phi(a+b)
 =\frac1{n^2}\sum_{h\in H}\phi(h)
 =\delta_{x,x'}\delta_{y,y'}.
$$
Thus $Q$ and $R$ are orthogonal.
\end{proof}

\begin{construction}\label{con:projection-product}
Let $K$ and $L$ be finite abelian groups with $|K|,|L|>2$. Let $n=|K||L|$.
For each $S\in\{K,L\}$, suppose that there exist normalized
orthogonal orthomorphisms $\theta_S,\eta_S:S\to S$ such that
\begin{enumerate}
\item[(i)] $\eta_S\in\operatorname{Aut}(S)$;
\item[(ii)]
$\langle\{(\theta_S(x),x):x\in S\}\rangle=S\times S$.
\end{enumerate}
Then there exists an $\mathrm{MCOQLS}(n)$.
\end{construction}

\begin{proof}
Let $\theta_K,\eta_K$ and $\theta_L,\eta_L$ be the normalized
orthogonal orthomorphisms on $K$ and $L$, respectively, satisfying
conditions~(i) and~(ii). Set $G=K\times L$ and define
$\theta_G(u,v)=(\theta_K(u),\theta_L(v))$ and
$\eta_G(u,v)=(\eta_K(u),\eta_L(v))$.
Then $\theta_G,\eta_G$ are normalized orthogonal orthomorphisms,
and $\eta_G$ is an automorphism.
We construct a normalized full-rank factorization of $G\times G$
and apply an automorphism to obtain the projection properties
required in Construction~\ref{con:projections}.

Set $E=(K\times\{0_L\})\cup(\{0_K\}\times L)$ and
$$
A=(E\times\{0_G\})\cup(\{0_G\}\times(G\setminus E)),
\qquad B=\{(\theta_G(y),y):y\in G\}.
$$
Both factors contain $(0_G,0_G)$, and $|A|=|B|=|G|$.
Since $\theta_G$ acts coordinatewise by permutations,
$x-\theta_G(y)\in E$ if and only if $y-\theta_G^{-1}(x)\in E$.
Thus every $(x,y)\in G\times G$ belongs to $A+B$, using
$$
(x,y)=
\begin{cases}
(x-\theta_G(y),0_G)+(\theta_G(y),y),&x-\theta_G(y)\in E,\\
(0_G,y-\theta_G^{-1}(x))+(x,\theta_G^{-1}(x)),&x-\theta_G(y)\notin E.
\end{cases}
$$
As $|A||B|=|G\times G|$, the representation is unique,
so $(A,B)$ is a normalized factorization.

Since $E$ generates $G$, the subgroup $\langle A\rangle$
contains $G\times\{0_G\}$.
For any $u\in K$, choose $w\in K\setminus\{0_K,u\}$
and $v\in L\setminus\{0_L\}$.
Both $(0_G,(w,v))$ and $(0_G,(w-u,v))$ belong to $A$,
and their difference is $(0_G,(u,0_L))$.
Interchanging $K$ and $L$ gives every $(0_G,(0_K,v))$.
Thus $\langle A\rangle$
contains $\{0_G\}\times G$ and $A$ generates $G\times G$.

By normalization, $B$ contains
$
((\theta_K(u),0_L),(u,0_L)),
((0_K,\theta_L(v)),(0_K,v))
$
for all $u\in K$ and $v\in L$.
By the hypotheses, these two sets generate
$(K\times\{0_L\})\times(K\times\{0_L\})$ and
$(\{0_K\}\times L)\times(\{0_K\}\times L)$, respectively.
Together they generate $G\times G$.
Hence $(A,B)$ is a normalized full-rank factorization.

	Define a homomorphism $\Phi:G\times G\to G\times G$ by
$
\Phi(x,y)=(y-x,\eta_G(y)-x).
$
Since $\eta_G$ is both an automorphism and an orthomorphism,
$\eta_G-\mathrm{id}_G$ is an automorphism.
For each $(z,w)\in G\times G$, the equation
$\Phi(x,y)=(z,w)$ has the unique solution
$y=(\eta_G-\mathrm{id}_G)^{-1}(w-z)$ and $x=y-z$.
Thus $\Phi$ is an automorphism.
By Lemma~\ref{lem:automorphism-factorization},
$(\Phi(A),\Phi(B))$ is a normalized full-rank factorization
of $G\times G$.

Since $E=-E$, we have
$$
\begin{aligned}
	\Phi(A)
	&=\{(e,e):e\in E\}
	\cup\{(y,\eta_G(y)):y\in G\setminus E\},\\
	\Phi(B)
	&=\{(y-\theta_G(y),\eta_G(y)-\theta_G(y)):y\in G\}.
\end{aligned}
$$
Since $\eta_G$ maps both $E$ and $G\setminus E$ bijectively onto themselves, both $\pi_1$ and $\pi_2$
restrict to bijections from $\Phi(A)$ onto $G$.
The same holds for $\Phi(B)$ because
$\mathrm{id}_G-\theta_G$ and $\eta_G-\theta_G$ are permutations.
Construction~\ref{con:projections} therefore gives an
$\mathrm{MCOQLS}(n)$.
\end{proof}

\begin{construction}\label{con:endomorphism}
Let $(A,B)$ be a normalized full-rank factorization of a finite abelian
group $G$ of order $n$. Suppose there exists an endomorphism $T:G\to G$
such that the following conditions are satisfied:
\begin{enumerate}
\item[(i)] $\mathrm{id}_{G}-T$ is an automorphism of $G$;
\item[(ii)] The maps $(a,b)\mapsto a+Tb$ and $(a,b)\mapsto Ta+b$
from $A\times B$ to $G$ are bijective.
\end{enumerate}
Then there exists an $\mathrm{MCOQLS}(n)$.
\end{construction}

\begin{proof}
	Since $(A,B)$ is a factorization of $G$, every $(x,y)\in G\times G$
	has a unique representation
	$(x,y)=(a,b)+(b',a'),$ with $a,a'\in A$ and $b,b'\in B$.
	Thus $(A\times B,B\times A)$ is a factorization of $G\times G$.
	Both factors contain $(0_G,0_G)$.
	Moreover, $A\times B$ contains $A\times\{0_G\}$ and
	$\{0_G\}\times B$, which together generate $G\times G$.
	The same argument applies to $B\times A$.
	Hence this factorization is normalized and has full rank.
	
	Define a homomorphism $\Phi:G\times G\to G\times G$ by
	$
	\Phi(u,v)=(u+v,u+Tv).
	$
	For each $(x,y)\in G\times G$, the equation
	$\Phi(u,v)=(x,y)$ has the unique solution
	$
	v=(\mathrm{id}_G-T)^{-1}(x-y), u=x-v.
	$
	Thus $\Phi$ is an automorphism.
	By Lemma~\ref{lem:automorphism-factorization},
	$(\Phi(A\times B),\Phi(B\times A))$ is a normalized
	full-rank factorization of $G\times G$.
	
	For $a\in A$ and $b\in B$,
	$
	\Phi(a,b)=(a+b,a+Tb),
	\Phi(b,a)=(a+b,Ta+b).
	$
	The factorization property of $(A,B)$ and condition~(ii) ensure that
	$\pi_1$ and $\pi_2$ restrict to bijections from each of
	$\Phi(A\times B)$ and $\Phi(B\times A)$ onto $G$.
	Construction~\ref{con:projections} therefore gives an
	$\mathrm{MCOQLS}(n)$.
\end{proof}

\section{Existence results}\label{sec:existence}
We establish two base existence results and then use the product
construction to prove Theorem~\ref{thm:main}.
\begin{lemma}\label{lem:field-permutations}
Let $q\geq9$ be an odd prime power, or let $q=16$.
The additive group of $\mathbb F_q$ satisfies the conditions
imposed on $S$ in Construction~\ref{con:projection-product}.
\end{lemma}

\begin{proof}
Take $M,\theta,\eta,a,b,c$ as in Lemma~\ref{lem:near-linear}.
The maps $\theta,\eta$ are normalized orthogonal orthomorphisms.
Since $\eta(x)=cx$ with $c\ne0$, $\eta$ is an automorphism
of the additive group of $\mathbb F_q$.
It remains to prove that
$\{(\theta(x),x):x\in\mathbb F_q\}$ generates
$(\mathbb F_q,+)\times(\mathbb F_q,+)$.

We first show that $M$ generates the additive group of $\mathbb F_q$.
If $q=p^e\geq9$ is odd, where $p$ is a prime, every proper additive subgroup of
$\mathbb F_q$ has order at most $q/p<(q-1)/2=|M|$.
Hence $\langle M\rangle=(\mathbb F_q,+)$.

For $q=16$, let $\zeta$ generate the multiplicative group $M$.
Since $\zeta$ has order $5$, it lies in no proper subfield
of $\mathbb F_{16}$.
Hence $\mathbb F_2(\zeta)=\mathbb F_{16}$.
As the powers of $\zeta$ belong to $M$, their
$\mathbb F_2$-linear span is $\mathbb F_{16}$.
Thus $M$ generates the additive group of $\mathbb F_{16}$.

The same holds for every multiplicative coset $tM$,
since multiplication by $t\ne0$ is an additive automorphism.

Let $N=\left\langle\{(\theta(x),x):x\in\mathbb F_q\}\right\rangle$.
Choose $t\in\mathbb F_q^\times\setminus M$.
Since $\theta(x)=ax$ on $M$ and $\theta(x)=bx$ on $tM$,
the additive generation by these two sets gives
$$
\{(ax,x):x\in\mathbb F_q\}\subset N, \quad
\{(bx,x):x\in\mathbb F_q\}\subset N.
$$
As $a\ne b$, we have $$\{((a-b)x,0):x\in \mathbb{F}_q\}=\mathbb{F}_q \times \{0\} \subset N, \quad \{(0,(b^{-1}a-1)x):x\in \mathbb{F}_q\}= \{0\} \times \mathbb{F}_q  \subset N$$. Thus $N=\mathbb{F}_q \times \mathbb{F}_q$.
This verifies all the conditions imposed on $S$
in Construction~\ref{con:projection-product}.
\end{proof}

\begin{theorem}\label{thm:field-products}
	Let $q_1$ and $q_2$ each be either an odd prime power at least $9$
	or $16$. Then there exists an $\mathrm{MCOQLS}(q_1q_2)$.
\end{theorem}
\begin{proof}
	By Lemma~\ref{lem:field-permutations}, the additive groups of
	$\mathbb F_{q_1}$ and $\mathbb F_{q_2}$ satisfy the hypotheses of Construction~\ref{con:projection-product}.
	The result follows from the construction.
\end{proof}

We use the normalized full-rank factorization of
$\mathbb Z_4^2\times\mathbb Z_2^2$ given by
Haanp\"a\"a, \"Osterg{\aa}rd and Szab\'o~\cite[Theorem~14]{HOS}.

\begin{lemma}\label{thm:even-orders}
There exist two orthogonal QLS$(n)$s, each with maximal cardinality,
for $n=64$ and $n=128$.
\end{lemma}
\begin{proof}
Take $G_0=\mathbb Z_4^2\times\mathbb Z_2^2$.
The sets
$A_0=\{a_i:i\in[8]\}$ and $B_0=\{b_i:i\in[8]\}$ listed below
form a normalized full-rank factorization of $G_0$.
Each string lists the coordinates of an element in order.
\begin{equation}\label{eq:factor64}
\begin{array}{c|cccccccc}
i&0&1&2&3&4&5&6&7\\ \hline
a_i&0000&0001&0100&0301&1000&1110&3001&3311\\
b_i&0000&0010&1200&1311&2110&2300&3101&3210
\end{array}
\end{equation}
Define an endomorphism $T_0:G_0\to G_0$ by
$$
 T_0(g_1,g_2,g_3,g_4)
 =(2g_1,2g_2,g_3+g_4,g_3+g_4).
$$
All coordinates are reduced modulo their respective group orders.
Since $T_0^2=0$, we have
$$
 (\mathrm{id}_{G_0}-T_0)(\mathrm{id}_{G_0}+T_0)=(\mathrm{id}_{G_0}+T_0)(\mathrm{id}_{G_0}-T_0)=\mathrm{id}_{G_0}-T_0^2=\mathrm{id}_{G_0}.
$$
Thus $\mathrm{id}_{G_0}-T_0$ is an automorphism with inverse $\mathrm{id}_{G_0}+T_0$.
The addition table for $A_0+T_0(B_0)$ is given below,
with rows indexed by $a_i$ and columns by $T_0(b_j)$,
in the order of~\eqref{eq:factor64}.
$$
\begin{array}{c|cccccccc}
a_i\backslash T_0(b_j)&0000&0011&2000&2200&0211&0200&2211&2011\\ \hline
0000&0000&0011&2000&2200&0211&0200&2211&2011\\
0001&0001&0010&2001&2201&0210&0201&2210&2010\\
0100&0100&0111&2100&2300&0311&0300&2311&2111\\
0301&0301&0310&2301&2101&0110&0101&2110&2310\\
1000&1000&1011&3000&3200&1211&1200&3211&3011\\
1110&1110&1101&3110&3310&1301&1310&3301&3101\\
3001&3001&3010&1001&1201&3210&3201&1210&1010\\
3311&3311&3300&1311&1111&3100&3111&1100&1300\\
\end{array}
$$
The $64$ entries are pairwise distinct and hence exhaust $G_0$.
Thus $(a,b)\mapsto a+T_0b$ is a bijection from $A_0\times B_0$
to $G_0$. Similarly, direct calculation shows that the $64$ sums
$T_0a+b$, with $(a,b)\in A_0\times B_0$, are pairwise distinct.
Therefore, $(a,b)\mapsto T_0a+b$ is also bijective,
and condition~(ii) of Construction~\ref{con:endomorphism} holds.
Thus $(A_0,B_0)$ and $T_0$ satisfy all the required conditions.

For order $128$, take $G_1=\mathbb Z_4^2\times\mathbb Z_2^3$ and set
$$
\begin{aligned}
 A_1&=\{(a_i,z):i\in[8],\ z\in\mathbb Z_2\},\\
 B_1&=\{00000,00100,12000,13110,21100,23000,31011,32101\}.
\end{aligned}
$$
Define an endomorphism $T_1:G_1\to G_1$ by
$$
 T_1(g_1,g_2,g_3,g_4,z)
 =(2g_1,2g_2+2z,g_3+g_4,g_3+g_4,g_2\bmod2).
$$
Here $0\in A_1\cap B_1$, $|A_1|=16$, and $|B_1|=8$.
Direct calculation gives $\langle A_1\rangle=\langle B_1\rangle=G_1$.
Similarly, computing $a+b$, $a+T_1b$ and $T_1a+b$
for all $128$ pairs $(a,b)\in A_1\times B_1$ gives $128$
pairwise distinct elements of $G_1$ for each of the three sum maps.
Thus all three maps are bijective.
In particular, $(A_1,B_1)$ is a normalized full-rank factorization
of $G_1$, and condition~(ii) of
Construction~\ref{con:endomorphism} holds.
Also,
$$
 T_1^2(g_1,g_2,g_3,g_4,z)=(0,2g_2,0,0,0),
 \qquad T_1^3=0.
$$
Hence $\mathrm{id}_{G_1}-T_1$ has inverse $\mathrm{id}_{G_1}+T_1+T_1^2$.
Thus $(A_1,B_1)$ and $T_1$ also satisfy all the required conditions.

Applying Construction~\ref{con:endomorphism} to the two sets of data
gives the required pairs of orders $64$ and $128$.
\end{proof}

We can now combine the preceding existence results with the product construction to obtain the main existence theorem.
\begin{proof}[Proof of Theorem~\ref{thm:main}]
	For case~(i), Theorem~\ref{thm:field-products} gives an
	$\mathrm{MCOQLS}(q_1q_2)$. The case $k=1$ follows immediately.
	For $k>1$ with $k\notin\{2,6\}$, Construction~\ref{con:product}
	gives an $\mathrm{MCOQLS}(q_1q_2k)$.
	
	For case~(ii), Lemma~\ref{thm:even-orders} gives
	$\mathrm{MCOQLS}(64)$ and $\mathrm{MCOQLS}(128)$.
	The pair of order $64$, together with Construction~\ref{con:product},
	covers every $m\notin\{2,6\}$, including the trivial case $m=1$.
	For $m=2$, use the pair of order $128$.
	For $m=6$, apply Construction~\ref{con:product} to the pair of
	order $128$ with multiplier $3$. This completes the proof.
\end{proof}

\section{Conclusion}\label{sec:conclusion}

We have constructed MCOQLSs from normalized full-rank factorizations of finite abelian groups. To obtain the factorizations required in Construction~\ref{con:projections}, we developed two approaches: one based on orthogonal orthomorphisms and the other on suitable endomorphisms. Both approaches apply more broadly than the cases considered in Section~\ref{sec:existence}. For example, let $G$ be a finite abelian group of odd order and let $(A,B)$ be a normalized full-rank factorization of $G$. Then $T=-\mathrm{id}_G$ satisfies the conditions of Construction~\ref{con:endomorphism}. Indeed, $\mathrm{id}_G-T=2\mathrm{id}_G$ is an automorphism, while the factorization property implies that the maps $(a,b)\mapsto a-b$ and $(a,b)\mapsto b-a$ are bijective. Thus, the full-rank factorizations given by Dinitz~\cite{Dinitz} provide further applications of this approach, although the corresponding orders considered here are already covered by Theorem~\ref{thm:main}.

The two families of orders in Theorem~\ref{thm:main} overlap, but neither contains the other. The bound on the number of mutually orthogonal QLSs excludes an $\mathrm{MCOQLS}(2)$.
Moreover, the results of Ball and Simoens~\cite{BS6} imply that
every orthogonal pair of order $3$, $4$ or $5$ is classical,
and that no orthogonal pair of order $6$ exists. The smallest order obtained by our constructions is $64$. This leaves two natural questions: What is the smallest order $n \ge 2$ for which an $\mathrm{MCOQLS}(n)$ exists? Does there exist an integer $n_0$ such that an $\mathrm{MCOQLS}(n)$ exists for every $n\geq n_0$?

\section*{Acknowledgements}
\noindent H. Cao's research was supported by the National Natural Science Foundation of China (Grants No. 12471313 and No. 12071226). Y.Zhang's research was supported by the Postgraduate Research \& Practice Innovation Program of Jiangsu Province (No.~26CXJH3078).

\section*{Conflicts of Interest Statement}
\noindent The authors declare no conflicts of interest.

\end{document}